%% file: LensGrothendieck.tex
\documentclass[11pt, oneside, article]{memoir}

\settrims{0pt}{0pt} 
\settypeblocksize{*}{33pc}{*} 
\setlrmargins{*}{*}{1} 
\setulmarginsandblock{1.2in}{1.5in}{*} 
\setheadfoot{\onelineskip}{2\onelineskip} 
\setheaderspaces{*}{1.5\onelineskip}{*} 
\checkandfixthelayout

\usepackage[centertags,sumlimits,intlimits,namelimits,reqno]{amsmath}
\usepackage{latexsym}
\usepackage{xcolor}
\usepackage{amsfonts,amsthm,amssymb}
\usepackage{mathtools}
\usepackage[cal=euler,scr=rsfso]{mathalfa}
\usepackage[boxslash]{stmaryrd}
\usepackage{newpxtext}
\usepackage{enumitem}
\usepackage{ifthen}
\usepackage[T1]{fontenc}
\RequirePackage{hyperref}
\hypersetup{colorlinks,linkcolor={blue!50!black},citecolor={red!50!black},urlcolor={blue!80!black}}
\usepackage{tikz}
\usepackage[capitalize]{cleveref}
\usepackage[backend=biber, maxbibnames = 10, style = alphabetic]{biblatex}
\DeclareMathAlphabet{\mathpzc}{OT1}{pzc}{m}{it}
\usepackage[color=white]{todonotes}
\usepackage{xstring}

\newcommand{\bigexists}{%
\mathop{\lower.9ex\hbox{%
   \scalebox{1.9}{\ensuremath{\exists}}}}\limits}

\DeclareFontFamily{U}{mathx}{\hyphenchar\font45}
\DeclareFontShape{U}{mathx}{m}{n}{
      <5> <6> <7> <8> <9> <10>
      <10.95> <12> <14.4> <17.28> <20.74> <24.88>
      mathx10
      }{}
\DeclareSymbolFont{mathx}{U}{mathx}{m}{n}
\DeclareFontSubstitution{U}{mathx}{m}{n}
\DeclareMathAccent{\widecheck}{0}{mathx}{"71}

  \definecolor{darkblue}{rgb}{0,0,0.7} 

	\allowdisplaybreaks
	
  \usetikzlibrary{ 
  	cd,
  	math,
  	decorations.markings,
		decorations.pathreplacing,
  	positioning,
	 	circuits.logic.US,
 		arrows.meta,
  	shapes,
		shadows,
		shadings,
  	calc,
  	fit,
  	quotes,
  	intersections,
  }

\input{tikz-stuff20190803.tex}

  \setlist{noitemsep, nolistsep}
	\setlist[description]{leftmargin=0em, itemindent=2em}
	
\theoremstyle{plain}
\newtheorem{theorem}{Theorem}[chapter] 
\newtheorem{proposition}[theorem]{Proposition}

\theoremstyle{definition}
\newtheorem{definition}[theorem]{Definition}

\newtheorem{notation}[theorem]{Notation}

\newtheorem*{axiom*}{Axiom}

\theoremstyle{remark}
\newtheorem{example}[theorem]{Example}
\newtheorem{remark}[theorem]{Remark}

\crefalias{chapter}{section}

\newcommand{\Set}[1]{\mathrm{#1}}
\newcommand{\cat}[1]{\mathcal{#1}}
\newcommand{\Cat}[1]{{\mathsf{#1}}}
\newcommand{\Funr}[1]{\mathsf{#1}}

\DeclareMathOperator{\ob}{\Set{Ob}}

\DeclareMathOperator{\Hom}{Hom}
\DeclareMathOperator{\id}{id}

\newcommand{\finset}{\Cat{FinSet}}
\newcommand{\smset}{\Cat{Set}}
\newcommand{\smcat}{\Cat{Cat}}
\newcommand{\gr}[1]{\Cat{Gr}(#1)}
\newcommand{\grc}[1]{\Cat{Gr}^\mathrm{o}(#1)}

\newcommand{\pb}[1][very near start]{\ar[dr, phantom, #1, "\lrcorner"]}

\newcommand{\rr}{{\mathbb{R}}}

\newcommand{\imp}{\Rightarrow}

\newcommand{\op}{^\mathrm{op}}
\newcommand{\po}{^\mathrm{p}}
\newcommand{\shp}{^\sharp}

\newcommand{\cp}{\mathbin{\fatsemi}}
\newcommand{\qqand}{\qquad\text{and}\qquad}

\newcommand{\lens}{\Cat{Lens}}
\newcommand{\lo}[2]{\begin{psmallmatrix}#1\\#2\end{psmallmatrix}}
\newcommand{\comon}{\Cat{Comon}}
\newcommand{\slice}{\Cat{Slice}}
\newcommand{\coslice}{\Cat{Coslice}}
\newcommand{\coKl}{\Funr{coKl}}
\newcommand{\Top}{\Cat{Top}}

\begin{document}   

\title{Generalized Lens Categories via Functors $\cat{C}\op\to\smcat$}

\author{David I.\ Spivak\thanks{This work supported by Honeywell Inc.\ as well as AFOSR grants FA9550-17-1-0058 and FA9550-19-1-0113.}}
\date{\vspace{-.3in}}

\maketitle

\begin{abstract}
Lenses have a rich history and have recently received a great deal of attention from applied category theorists. We generalize the notion of lens by defining a category $\lens_F$ for any category $\cat{C}$ and functor $F\colon\cat{C}\op\to\smcat$, i.e.\ an indexed category, using a variant of the Grothendieck construction. All of the mathematics in this note is straightforward; the purpose is simply to see lenses in a broader context where some closely-related examples, such as ringed spaces and open continuous dynamical systems, can be included.
\end{abstract}

\chapter{Introduction}

Roughly speaking, a \emph{lens} is bi-directional map $\lo{\mathrm{get}}{\mathrm{put}}\colon\lo{c}{x}\to\lo{d}{y}$ between pairs; the two parts have the following form:
\begin{equation}\label{eqn.basic_form}
	\mathrm{get}\colon c\to d
	\qqand
	\mathrm{put}\colon c\times y\to x.
\end{equation}
Lenses have recently received a great deal of attention from applied category theorists. One reason is that they show up in many disparate places, such as database updates, learning algorithms, open games, open dynamical systems, and wiring diagrams. Lenses have been broadly generalized to so-called profunctor optics; see e.g.\ \cite{riley2018categories}. 

We will discuss what seems to be a completely different direction of generalization: we associate a notion of \emph{generalized lens category} to an arbitrary indexed category. Namely for every category $\cat{C}$ and (pseudo-) functor $F\colon \cat{C}\op\to\smcat$, we define a category $\lens_F$ using a variant of the Grothendieck construction. The idea is that a morphism in the Grothendieck construction consists of two parts, which turn out to be the above get and put maps from lens theory. Taking $\cat{C}=\smset$, we can recover the usual category of lenses in a couple of ways (see \cref{ex.dependent_lens,prop.recover_usual}). The most basic is to take $F$ to be the slice category functor $\slice\colon\smset\op\to\smcat$, embed each $\lo{c}{x}$ as the projection $\pi_c\colon c\times x\to c$ in $\slice(c)$, and note that for any choice of function $\mathrm{get}\colon c\to d$, the square shown below is a pullback:
\begin{equation}\label{eqn.lens_via_pb}
\begin{tikzcd}
	c\times x\ar[dr, bend right, "\pi_c"']&
	c\times y\ar[r]\ar[d]\pb\ar[l, "\mathrm{put}"']&
	d\times y\ar[d, "\pi_{d}"]\\&
	c\ar[r, "\mathrm{get}"']&
	d
\end{tikzcd}
\end{equation}
The function indicated as $\mathrm{put}$ does not have quite the same form as in \cref{eqn.basic_form}: there is an extra factor of $c$ in the codomain. However, to be a morphism in the slice category $\slice(c)$, such a function $c\times y\to c\times x$ in $\slice(c)$ must commute with the projections. Thus it has no choice with regards to the $c$-factor in the codomain, and hence the only remaining choice is that of a function $\mathrm{put}\colon c\times y\to x$, thus recovering the notion of morphism $\lo{c}{x}\to\lo{d}{y}$ from \cref{eqn.basic_form}. This idea, to think of the an object $\lo{c}{x}$ not as a simple pair but as a \emph{dependent} pair ($x$ dependent on $c$), is the main thrust of this note.

All of the results we discuss here are straightforward to prove. The proposed contribution is to provide a setting in which open continuous dynamical systems, ringed spaces, and dependent lenses---none of which fit in with the usual definition of lens but all of which seem to be quite similar to it in spirit---can be included in the theory. We also provide a construction of lenses in an arbitrary symmetric monoidal category, which is known but seems not to have been written down explicitly before.

\paragraph{Acknowledgements.}
Thanks to Brendan Fong, Bruno Gavranovi\'{c}, David A.\ Dalrymple, Sophie Libkind, Eliana Lorch, and Toby St Clere Smithe for inspiring conversations. Special thanks to David Jaz Myers for \cref{sec.lenses_SMC}---which is entirely due to him---as well as other useful comments and suggestions, and to Christina Vasilakopoulou for the observation that wiring diagrams in \cite{Vagner.Spivak.Lerman:2015a,schultz2016dynamical} are prisms (\cref{ex.prisms_wds}). Thanks also to Bruno Gavranovi\'{c} and Sophie Libkind for a careful reading, and to Jules Hedges for several references and helpful comments.

\chapter{Lenses}

Lenses have been studied in computer science and discussed category-theoretically for several decades \cite{oles1983category,dePaiva1989dialectica,bohannon2006relational, diskin2008algebraic, johnson2012lenses}. There are several variants, and the naming is often inconsistent. A good summary can be found in this \href{https://julesh.com/2018/08/16/lenses-for-philosophers/}{blog post} by Jules Hedges; see also the \href{http://hackage.haskell.org/package/lens}{Haskell Lens library}. We will be discussing what Hedges calls bimorphic lenses in \cite{hedges2017coherence}, but we will refer to them simply as \emph{lenses}. We begin by recalling this notion.

\section{Lenses in finite product categories}

We begin with lenses in a category $\cat{C}$ with finite products; one may think $\cat{C}=\smset$. In \cref{sec.lenses_SMC} we generalize to an arbitrary symmetric monoidal category, and then much further in \cref{chap.generalized_lenses}.

\begin{notation}
We denote the composite of $f\colon c\to d$ and $g\colon d\to e$ by $(f\cp g)\colon c\to e$. We denote the identity morphism on an object $c$ either by $\id_c$ or simply by $c$. We denote the hom-set from $c$ to $d$ in a category $\cat{C}$ either by $\Hom(c,d)$, $\Hom_{\cat{C}}(c,d)$, or $\cat{C}(c,d)$.
\end{notation}

\begin{definition}[Lenses in finite product categories]\label{def.lens_fpcat}
Let $\cat{C}$ be a category with finite products. The \emph{category of $\cat{C}$-lenses}, denoted $\lens_{\cat{C},\times}$ has as objects pairs $\lo{c}{x}$, where $c,x\in\cat{C}$; given another such object $\lo{d}{y}$ we define the homset
\[
	\lens_{\cat{C},\times}\left(\lo{c}{x},\lo{d}{y}\right)\coloneqq
	\left\{\lo{f}{f\shp}\;\middle|\; f\colon c\to d\text{ and }f\shp\colon c\times y\to x\right\}.
\]
We refer to the $\cat{C}$-morphisms $f\colon c\to d$ as the \emph{get part} and $f\shp\colon c\times y\to x$ as the the \emph{put part} of the \emph{lens} $\lo{f}{f\shp}$.

The identity on $\lo{c}{x}$ is $\lo{\id_c}{\epsilon_c\times \id_x}=\lo{c}{\epsilon_c\times x}$, where $\epsilon_c\colon c\to 1$ is the terminal map. The composite of $\lo{f}{f\shp}\colon\lo{c}{x}\to\lo{d}{y}$ and $\lo{g}{g\shp}\colon\lo{d}{y}\to\lo{e}{z}$ is $\lo{f\cp g}{(\delta_c\times z)\cp(c\times f\times z)\cp(c\times g\shp)\cp f\shp}$, where $\delta_c\colon c\to c\times c$ is the diagonal. In pictures, it is given by the following string diagram:
\begin{equation}\label{eqn.string}
\begin{tikzpicture}[WD, bb port length=5pt, baseline=(f)]
	\node[bb={1}{1}] (f) {$f$};
	\node[bb={1}{1}, right=1.5 of f] (g) {$g$};
	\node[left=0 of f_in1] {$c$};
	\node[right=0 of g_out1] {$e$};
	\draw (f_out1) to node[above, font=\scriptsize] {$d$} (g_in1);
	\node[dot, right=6 of g] (dot) {};
	\draw (dot.180) to +(180:5pt) coordinate (dot_in1) node[left] {$c$};
	\node[bb={1}{1}] at ($(dot)+(1.5,-.5)$) (f) {$f$};
	\node[bb={2}{1}, below right=-1 and 1.5 of f] (g') {$g\shp$};
	\draw (dot) to[out=-60, in=180] (f_in1);
	\draw (f_out1) to node[above, font=\scriptsize] {$d$} (g'_in1);
	\node[bb={2}{1}, above right=-1 and 1.5 of g'] (f') {$f\shp$};
	\draw (dot) to [out=60, in=180] (f'_in1);
	\draw (g'_out1) to[out=0, in=180] node[below, font=\scriptsize] {$y$} (f'_in2);
	\draw (g'_in2) to (g'_in2-|dot_in1) node[left] {$z$};
	\node[right=0 of f'_out1] {$x$};
\end{tikzpicture}
\end{equation}
\end{definition}

\begin{example}[Simple lenses]\label{ex.simple}
In functional programming the emphasis has often been on what we call \emph{simple lenses} \cite{bohannon2006relational}, which are lenses of the form $\lo{c}{c}$. A morphism of simple lenses $\lo{c}{c}\to\lo{d}{d}$ consists of a function $f\colon c\to d$ and a function $f\shp\colon c\times d\to c$.
\end{example}

\begin{example}[Prisms, wiring diagrams]\label{ex.prisms_wds}
If $\cat{C}$ has finite coproducts then $\cat{C}\op$ has finite products. The category $(\lens_{\cat{C}\op,+})\op$ is called the category of \emph{prisms} in $\cat{C}$. A prism $\lo{c}{x}\to\lo{d}{y}$ consists of a pair of morphisms $d\to c$ and morphism $x\to c+y$ in $\cat{C}$.

The category $(\lens_{\finset\op,+})\op$ of prisms in $\finset$, or more generally replacing $\finset$ by $\finset/T$ for some set $T$, is called the category of wiring diagrams in \cite{Vagner.Spivak.Lerman:2015a}. When $T=1$ this forms the left class of a factorization system on the category $\Cat{Cob}$ of 1-dimensional oriented cobordisms, and similarly for arbitrary $T$ (using cobordisms with components labeled in $T$). See \cite{abadi2015existence}.
\end{example}

\begin{example}[Moore machines]
Given a pair of sets $(A,B)$, a Moore machine (also called an open discrete dynamical system in \cite{Spivak:2015a}) consists of a set $S$ and two functions $f^{\mathrm{out}}\colon S\to B$ and $f^{\mathrm{upd}}\colon A\times S\to S$. This is the same as a lens $\lo{S}{S}\to\lo{A}{B}$.

The notion of dynamical system (and the formula for composing them with wiring diagrams as in \cref{ex.prisms_wds}) can be generalized to the continuous case, with manifolds replacing the sets and systems of ordinary differential equations replacing the update functions. However, the theory of lenses does not accommodate this generalization. Remedying this lack was in part the motivation for the present note; see \cref{ex.cts_dyn}.
\end{example}

\section{Lenses in symmetric monoidal categories}\label{sec.lenses_SMC}

We owe the ideas of this section to David Jaz Myers, though these ideas are apparently folklore. For something similar, see \cite[Section 2.2]{abou2016reflections}.

\begin{definition}[Commutative comonoid]\label{def.comonoid}
Let $(\cat{C},I,\otimes,\sigma)$ be a symmetric monoidal category. A \emph{commutative comonoid in $\cat{C}$} consists of a tuple $(c,\epsilon,\delta)$ where $c\in\cat{C}$, $\epsilon\colon c\to I$ and $\delta\colon c\to c\otimes c)$ satisfy the axioms:
\begin{enumerate}
	\item $\delta_c\cp(c\otimes\epsilon_c)=c$,
	\item $\delta_c\cp\sigma_{c,c}=\delta_c$ , and 
	\item $\delta_c\cp(\delta_c\otimes c)=\delta_c\cp(c\otimes\delta_c)$.
\end{enumerate}
We refer to $\epsilon$ as the \emph{counit} and $\delta$ as the \emph{comultiplication}. We sometimes write $c$ to denote the comonoid, leaving $\epsilon$ and $\delta$ implicit.

A \emph{morphism} of commutative comonoids $(c,\epsilon_c,\delta_c)\to(d,\epsilon_{d},\delta_{d})$ is a morphism $f\colon c\to d$ in $\cat{C}$ such that $\epsilon_c=f\cp\epsilon_{d}$ and $\delta_c\cp(f\otimes f)=f\cp\delta_{d}$. We denote the category of commutative comonoids and their morphisms by $\comon_{\cat{C}}$.
\end{definition}

\begin{proposition}[Finite product categories]\label{prop.fpcats_comonoids}
If $\cat{C}$ has finite products and $(I,\otimes)$ is the corresponding (``Cartesian'') monoidal structure, then there is an isomorphism of categories $\cat{C}\cong\comon_{\cat{C}}$.
\end{proposition}
\begin{proof}
See \cite{fox1976coalgebras}.
\end{proof}

The following is straightforward.
\begin{proposition}
There is a symmetric monoidal structure on $\comon_{\cat{C}}$ such that the functor $\comon_{\cat{C}}\to\cat{C}$ is strict monoidal.
\end{proposition}

\begin{definition}\label{def.lens_smc}
Let $(\cat{C},I,\otimes,\sigma)$ be a symmetric monoidal category. The \emph{category of $\cat{C}$-lenses}, denoted $\lens_{\cat{C},\otimes}$ has as objects pairs $\lo{c}{x}$, where $c\in\comon_{\cat{C}}$ and $x\in\cat{C}$. Given another such object $\lo{d}{y}$ we define the homset
\[
	\lens_{\cat{C},\otimes}\left(\lo{c}{x},\lo{d}{y}\right)\coloneqq
	\left\{\lo{f}{f\shp}\;\middle|\; f\in\comon_{\cat{C}}(c, d)\text{ and }f\shp\in\cat{C}(c\otimes y, x)\right\}.
\]
We refer to the comonoid morphism $f\colon c\to d$ as the \emph{get part} and the map $f\shp\colon c\otimes y\to d$ as the \emph{put part} of the \emph{lens} $\lo{f}{f\shp}$.

The identity on $\lo{c}{x}$ is $\lo{c}{\epsilon_c\otimes x}$, where $\epsilon_c\colon c\to 1$ is the counit. The composite of $\lo{f}{f\shp}$ and $\lo{g}{g\shp}\colon\lo{d}{y}\to\lo{e}{z}$ is $\lo{f\cp g}{(\delta_c\otimes z)\cp(c\otimes f\otimes z)\cp(c\otimes g\shp)\cp f\shp}$, where $\delta_c\colon c\to c\otimes c$ is the comultiplication. The string diagram for the composite is identical to that in \cref{eqn.string}.
\end{definition}

\cref{def.lens_smc} generalizes \cref{def.lens_fpcat}, by \cref{prop.fpcats_comonoids}.

\chapter{Generalized lens categories}\label{chap.generalized_lenses}

We define the lens category $\lens_F$ for any functor $F\colon\cat{C}\op\to\smcat$ and then give several examples.

\section{Definition and examples of $\lens_F$}

The Grothendieck construction comes in two variants.

\begin{definition}[Grothendieck constructions]\label{def.grothendieck}
Let $\cat{C}$ be a category and $F\colon\cat{C}\to\smcat$. The \emph{covariant Grothendieck construction of $F$} consists of a category $\gr{F}$ and a functor $\pi_F\colon\gr{F}\to\cat{C}$, defined as follows.
\begin{align*}
	\ob(\gr{F})&
	\coloneqq
	\bigsqcup_{c\in\ob(\cat{C})}\ob(F(c))\\
	\gr{F}\big((c,x),(d,y)\big)&
	\coloneqq
	\bigsqcup_{f\in\cat{C}(c,d)}\Hom_{F(d)}\big(F(f)(x),y\big)
\end{align*}
That is, an object in $\gr{F}$ is a pair $(c,x)$ where $c\in\cat{C}$ and $x\in F(c)$. A morphism $(c,x)\to (d,y)$ is a pair $(f,f\shp)$ where $f\colon c\to d$ and $f\shp\colon F(f)(x)\to y$ is a morphism in the category $F(d)$. The identity on $(c,x)$ is $(\id_c,\id_x)$, and the composite of $(f,f\shp)$ and $(g,g\shp)$ is given by
\[(f,f\shp)\cp (g,g\shp)\coloneqq \left((f\cp g), \big(F(g)(f\shp)\cp g\shp\big)\right).\]
The functor $\pi_F\colon\gr{F}\to\cat{C}$ sends $(c,x)\mapsto c$ and $(f,f\shp)\mapsto f$.

Given a functor $F\colon\cat{C}\op\to\smcat$, the \emph{contravariant Grothendieck construction of $F$} consists of a category $\grc{F}$ and a functor $\pi_F\colon\grc{F}\to\cat{C}$, defined as follows:
\begin{align*}
	\ob(\grc{F})&
	\coloneqq
	\bigsqcup_{c\in\ob(\cat{C})}\ob(F(c))\\
	\grc{F}\big((c,x),(d,y)\big)&
	\coloneqq
	\bigsqcup_{f\in\cat{C}(c,d)}\Hom_{F(c)}\big(x,F(f)(y)\big)
\end{align*}
Identities, composition, and the functor $\pi_F$ are defined analogously.
\end{definition}

For any functor $F\colon\cat{C}\op\to\smcat$, let $F\po\colon\cat{C}\op\to\smcat$ denote its pointwise opposite, $F\po(c)\coloneqq F(c)\op$. The following is straightforward.

\begin{proposition}\label{prop.tfae}
Let $F\colon\cat{C}\op\to\smcat$ be a functor, and let $F\po\colon\cat{C}\op\to\smcat$ be its pointwise opposite. The following three categories are naturally isomorphic:
\begin{enumerate}
	\item $\gr{F}\op$, the opposite of the covariant Grothendieck construction of $F$,
	\item $\grc{F\po}$, the contravariant Grothendieck construction of the pointwise opposite of $F$,
	\item the analogous category with objects $\bigsqcup_{c\in\ob(\cat{C})}\ob(F(c))$ and morphisms
	\begin{equation}\label{eqn.lens_explicit}
		\Hom((c,x),(d,y))\coloneqq\bigsqcup_{f\in\cat{C}(c,d)}\Hom_{F(c)}\big(F(f)(y),x\big).
	\end{equation}
\end{enumerate}
Moreover, these isomorphisms commute with the functors $\pi_F$ to $\cat{C}$.
\end{proposition}

%
\begin{definition}[$F$-lenses]\label{def.F_lens}
Let $F\colon\cat{C}\op\to\smcat$ be a functor. Define the \emph{category of $F$-lenses}, denoted $\lens_F$, to be any of the three isomorphic categories from \cref{prop.tfae}. We refer to $\cat{C}$ as the \emph{get-category} and to $\pi_F\colon\lens_F\to\cat{C}$ as the \emph{get-functor}.

We denote the object $(c,x)$ by $\lo{c}{x}$. From the explicit formula \eqref{eqn.lens_explicit}, we see that a morphism $\lo{f}{f\shp}\colon\lo{c}{x}\to\lo{d}{y}$ consists of a pair $(f,f\shp)$ where $f\colon c\to d$ is in $\cat{C}$ and $f\shp\colon F(f)(y)\to x$ is a morphism in the category $F(c)$.
\end{definition}

With this definition, lenses have a diagrammatically simpler form than in \cref{eqn.string}. Here we show morphisms $\lo{f}{f\shp}\colon\lo{c}{x}\to\lo{d}{y}$ and $\lo{g}{g\shp}\colon\lo{d}{y}\to\lo{e}{z}$:
\[
\begin{tikzpicture}
	\node (P11) {
	\begin{tikzpicture}[WD, bb port length=.3cm]
		\node[bb={1}{1}] (f) {$f$};
  	\node[shell, right=.5cm of f] (y) {$y$};
		\draw (f_out1) -- (y);
		\node[above=0 of f_in1, font=\scriptsize] {$c$};
		\node[above=0 of f_out1, font=\scriptsize] {$d$};
	\end{tikzpicture}
	};
	\node (P12) [right=1 of P11] {
	\begin{tikzpicture}[WD, bb port length=.3cm]
		\node[shell] (x) {$x$};
		\draw (x.west) -- +(-.3cm, 0) node[above, font=\scriptsize] {$c$};
	\end{tikzpicture}
	};
	\node at ($(P11.east)!.5!(P12.west)$) {$\xRightarrow{f\shp}$};
	\node (P21) [right=1 of P12] {
	\begin{tikzpicture}[WD, bb port length=.3cm]
		\node[bb={1}{1}] (g) {$g$};
  	\node[shell, right=.6cm of g] (z) {$z$};
		\draw (g_out1) -- (z);
		\node[above=0 of g_in1, font=\scriptsize] {$d$};
		\node[above=0 of g_out1, font=\scriptsize] {$e$};
	\end{tikzpicture}
	};
	\node (P22) [right=1 of P21] {
	\begin{tikzpicture}[WD, bb port length=.3cm]
		\node[shell] (y) {$y$};
		\draw (y.west) -- +(-.3cm, 0) node[above, font=\scriptsize] {$d$};
	\end{tikzpicture}
	};
	\node at ($(P21.east)!.5!(P22.west)$) {$\xRightarrow{g\shp}$};
\end{tikzpicture}
\]
The wires represent objects $c,d,e\in\cat{C}$, and white the squares represent morphisms in $\cat{C}$. The blue circles and double-arrows represent objects and morphisms in the categories $F(c)$, etc. Here is a picture of the composite $\lo{f}{f\shp}\cp\lo{g}{g\shp}$:
\[
	\begin{tikzpicture}
		\node (P31) {
  	\begin{tikzpicture}[WD, bb port length=.3cm]
  		\node[bb={1}{1}] (f) {$f$};
  		\node[bb={1}{1}, right=.6cm of f] (g) {$g$};
    	\node[shell, right=.6cm of g] (z) {$z$};
  		\draw (g_out1) -- (z);
  		\node[above=0 of g_in1, font=\scriptsize] {$d$};
  		\node[above=0 of g_out1, font=\scriptsize] {$e$};
  		\draw (f_out1) -- (y);
  		\node[above=0 of f_in1, font=\scriptsize] {$c$};
  	\end{tikzpicture}	
  	};
		\node (P32) [right=1 of P31] {
  	\begin{tikzpicture}[WD, bb port length=.3cm]
  		\node[bb={1}{1}] (f) {$f$};
    	\node[shell, right=.5cm of f] (y) {$y$};
  		\draw (f_out1) -- (y);
  		\node[above=0 of f_in1, font=\scriptsize] {$c$};
  		\node[above=0 of f_out1, font=\scriptsize] {$d$};
  	\end{tikzpicture}
  	};
		\node (P33) [right=1 of P32] {
  	\begin{tikzpicture}[WD, bb port length=.3cm]
  		\node[shell] (x) {$x$};
  		\draw (x.west) -- +(-.3cm, 0) node[above, font=\scriptsize] {$c$};
  	\end{tikzpicture}
		};
  	\node at ($(P31.east)!.5!(P32.west)$) {$\xRightarrow{g\shp}$};
  	\node at ($(P32.east)!.5!(P33.west)$) {$\xRightarrow{f\shp}$};
	\end{tikzpicture}
\]
One may also imagine these morphisms logically, e.g.\ the implications
\[\forall c\ldotp y(f(c))\Rightarrow x(c)\qquad \forall d\ldotp z(g(d))\Rightarrow y(d)\]
which can be combined to obtain $\forall c\ldotp z(f(g(c)))\imp x(c)$.

\begin{remark}\label{rem.why_not_po}
The Grothendieck construction of a $\smcat$-valued functor $F\colon\cat{C}\op\to\smcat$ always yields a (split) fibration over $\cat{C}$ and vice versa, so generalized lens categories can be viewed simply as split fibrations. However we chose $\lens_F$ to be the fiberwise opposite of $\gr{F}$---rather than replacing $F$ with $F\po$ at the outset---for two reasons. First, in cases of interest $F$ seems to be simpler to specify than $F\po$. Second, the form of \eqref{eqn.lens_explicit} is the one that is most familiar in lens theory.
\end{remark}

\begin{example}[Dependent lenses]\label{ex.dependent_lens}
Let $\cat{C}$ be a category with pullbacks. There is a functor $\slice\colon\cat{C}\op\to\smcat$ given on an object $c$ by the slice category $\slice(c)\coloneqq\cat{C}/c$ over $c$, and on morphisms $f\colon c\leadsto b$ in $\cat{C}\op$ by pullback. That is, for every object $p\colon x\to c$ in $\slice(c)$ we obtain an object $\slice(f)(p)\in\cat{C}/b$ using the following pullback diagram in $\cat{C}$:
\[
\begin{tikzcd}
	b\times_cx\ar[r]\ar[d, "\slice(f)(p)"']\pb&
	x\ar[d, "p"]\\
	b\ar[r, "f"']&
	c
\end{tikzcd}
\]
This extends to morphisms in $\slice(c)$ using the universal property of pullbacks.

The category $\lens_{\slice}$ has as objects pairs $\lo{c}{p}$ where $p\colon x\to c$, and as morphisms pairs $\lo{f}{f\shp}$ where $f\colon c\to d$ and $f\shp\colon c\times_{d} y\to x$. We can think of objects in $\lens_{\slice}$ as dependent lenses; for example if $\cat{C}=\smset$, then each object $p\colon x\to c$ may assign non-isomorphic fibers to different elements of $c$.

Note that we can find the category of lenses from \cref{def.lens_fpcat} inside of $\lens_F$. Indeed, it is isomorphic to the full subcategory spanned by all pairs $\lo{c}{\pi}$ for which $\pi\colon c\times x\to c$ is the projection for some $x\in\cat{C}$. This was discussed in the introduction around \cref{eqn.lens_via_pb}. We will recover the category of lenses in a completely different way in \cref{prop.recover_usual}.

More importantly, the get functor $\pi_F\colon\lens_{\slice}\to\cat{C}$ is not only a fibration but a bifibration. Indeed, the functor $\slice(f)\colon\slice(d)\to\slice(c)$ has a left adjoint $\Sigma_f\colon\slice(c)\to\slice(d)$, which one may call the \emph{dependent sum along $f$}, for any morphism $f\colon c\to d$ in $\cat{C}$. The name ``dependent lens'' is actually most appropriate when $\cat{C}$ not only has pullbacks but is locally cartesian closed. This simply means that each $\slice(f)$ additionally has a right adjoint $\Pi_f\colon\slice(c)\to\slice(d)$, called the \emph{dependent product along $f$}. In this case $\pi_F$ is a trifibration.
\end{example}

\begin{example}[Open continuous dynamical systems]\label{ex.cts_dyn}
Recall that a differentiable manifold $M$ has a tangent bundle $TM$ and a submersion $\pi_M\colon TM\to M$. Given a pair of manifolds $(A,B)$, \cite{Vagner.Spivak.Lerman:2015a} defines an open continuous dynamical system with inputs $A$ and outputs $B$ to be a manifold $S$ (called the state space), a differentiable map $f^{\text{out}}\colon S\to B$, and a differentiable map $f^{\text{dyn}}\colon A\times S\to TS$ such that $f^{\text{dyn}}\cp\pi_S=\pi_2$. We can see this as a morphism in a generalized lens category as follows.

Consider the functor $\Funr{Subm}\colon\Cat{Mfd}\op\to\smcat$ sending each manifold $M$ to the category of submersions over $M$, and sending a differentiable map $f\colon M\to N$ to the pullback functor along $f$. Then an open continuous dynamical system with inputs $A$ and outputs $B$ consists of a morphism $\lo{S}{TS}\to\lo{B}{\pi_B}$ in $\lens_{\Funr{Subm}}$, where $\pi_B\colon A\times B\to B$ is the projection. \cite{Vagner.Spivak.Lerman:2015a} shows that continuous dynamical systems can be wired together using prisms, as in \cref{ex.prisms_wds}.
\end{example}

\begin{example}[Cotangent bundles]
Consider again the lens category for the functor $\Funr{Subm}\colon\Cat{Mfd}\op\to\smcat$, as in \cref{ex.cts_dyn}. Whereas in that example we considered objects given by tangent bundles and found that certain morphisms between them were continuous dynamical systems, here we consider objects given by cotangent bundles and find that we get lenses between them canonically, without making choices.

To begin, note that a differentiable manifold $M$ also has a cotangent bundle $T^*M$; its fiber over each point $m\in M$ is the dual to the tangent space there, i.e.\ it is the vector space $T^*_mM\coloneqq\Cat{Vect}(T_mM,\rr)$ of linear maps from the tangent space to the ground field $\rr$. Given a differentiable function $f\colon M\to N$, the derivative (Jacobian) defines a map $Tf\colon TM\to TN$, in particular over each point $m\in M$ a linear transformation $T_mM\to T_{f(m)}N$. This in turn induces a linear transformation $\Cat{Vect}(T_{f(m)}N,\rr)\to\Cat{Vect}(T_mM,\rr)$ for each point $m$, and these assemble into a morphism $f^\sharp\colon f^*(T^*N)\to T^*M$ of bundles over $M$. All together we have canonically obtained a morphism $\lo{f}{f^\sharp}\colon\lo{M}{T^*M}\to\lo{N}{T^*N}$ in the lens category $\lens_{\Funr{Subm}}$, and hence a functor $\Cat{Mfd}\to\lens_{\Funr{Subm}}$.
\end{example}

\begin{example}[Ringed spaces]\label{ex.ringed_sp}
The category of ringed spaces from algebraic geometry is an example of a generalized lens category. There is a functor $\Funr{Sh}\colon\Top\op\to\smcat$, where $\Top$ is the category of topological spaces and $\Funr{Sh}(X)$ is the category of sheaves of rings on $X$; given a map $f\colon X\to Y$ in $\Top$, there is a functor $f^*$ which sends a sheaf on $Y$ to a sheaf on $X$, hence defining $\Funr{Sh}$ on morphisms.

The category $\lens_{\Funr{Sh}}$ of $\Funr{Sh}$-lenses has as objects pairs $\lo{X}{\cat{O}_X}$ where $\cat{O}_X$ is a sheaf of rings on $X$. A morphism $\lo{X}{\cat{O}_X}\to\lo{Y}{\cat{O}_Y}$ is a pair $(f,f\shp)$ where $f\colon X\to Y$ is a map of topological spaces and $f\shp\colon f^*\cat{O}_Y\to\cat{O}_X$ is a map of sheaves of rings.
\end{example}

\begin{example}[Twisted arrow category]
Let $\cat{C}$ be a category and consider the coslice functor $\coslice\colon\cat{C}\op\to\smcat$ sending $c\mapsto c/\cat{C}$; on morphisms $f\colon d\to c$ the functor $\coslice(f)$ sends an object $x\colon c\to d\in\coslice(c)$ to the composite $f\cp x$. The category $\lens_\coslice$ is the twisted arrow category $\Cat{tw}(\cat{C})$ of $\cat{C}$. An object $\lo{c}{x}$ in $\lens_\coslice$ is a morphism $x\colon c\to d$, and a morphism $\lo{c}{x}\to\lo{d}{y}$ in $\lens_\coslice$ is a twisted square
\[
\begin{tikzcd}
	c\ar[r, "f"]\ar[d, "x"' pos=.4]&
	d\ar[d, "y" pos=.3]\\
	d&
	d'\ar[l, "f\shp"]
\end{tikzcd}
\]
\end{example}

\section{Lenses in symmetric monoidal categories}\label{sec.lens_smc_v2}

Let $(\cat{C},I,\otimes)$ be a symmetric monoidal category, and let $\comon_{\cat{C}}$ denote its (symmetric monoidal) category of commutative comonoids (see \cref{def.comonoid}). For each object $c\in\comon_{\cat{C}}$ there is a comonad on $\cat{C}$ given by $x\mapsto c\otimes x$; the counit is given by $\epsilon_c\otimes x$ and the comultiplication is given by $\delta_c\otimes x$, maps which are natural in $x$. Forming the coKleisli category gives a functor $\coKl_{\cat{C}}\colon\comon_{\cat{C}}\op\to\smcat$. Let's unpack this.

The functor $\coKl_{\cat{C}}\colon\comon_{\cat{C}}\op\to\smcat$ has the following more explicit formulation. Given an commutative comonoid $(c,\epsilon,\delta)$, the coKleisli category $\coKl_{\cat{C}}(c)$ has objects $\ob(\cat{C})$ and morphisms $\Hom(x,y)\coloneqq\cat{C}(c\otimes x,y)$. The identity on $x$ is given by $(\epsilon_c\otimes x)\colon c\otimes x\to x$, and the composite of $f\colon c\otimes x\to y$ and $g\colon c\otimes y\to z$ is given by
$(f\cp g)=(\delta_c\otimes x)\cp(c\otimes f)\cp g$. In pictures:
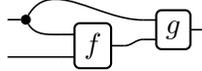
\[
\begin{tikzpicture}[WD, bb port length=5pt]
	\node[dot] (dot) {};
	\draw (dot.180) -- +(180:5pt) coordinate (dot_in1);
	\node[bb={2}{1}, below right=0 and 1.5 of dot] (f) {$f$};
	\draw (dot) to[out=-60, in=180] (f_in1);
	\node[bb={2}{1}, above right=-1 and 1.5 of f] (g) {$g$};
	\draw (dot) to [out=60, in=180] (g_in1);
	\draw (f_out1) to[out=0, in=180] (g_in2);
	\draw (f_in2) -- (f_in2-|dot_in1);
\end{tikzpicture}
\]
Given a morphism of comonoids $p\colon c\to d$ in $\comon_{\cat{C}}$, we obtain an identity-on-objects functor $\coKl_{\cat{C}}(p)\colon \coKl_{\cat{C}}(d)\to \coKl_{\cat{C}}(c)$ that sends the morphism $d\otimes x\to y$ to its composite with $(p\otimes x)\colon(c\otimes x)\to (d\otimes x)$.

\begin{proposition}\label{prop.recover_usual}
Let $\cat{C}$ be symmetric monoidal and let $\coKl\colon\comon_{\cat{C}}\op\to\smcat$ be as in \cref{sec.lens_smc_v2}. The generalized lens category $\lens_{\coKl}$ is isomorphic to the category $\lens_{\cat{C},\otimes}$ from \cref{def.lens_smc}.
\end{proposition}
\begin{proof}
The objects of both $\lens_{\coKl}$ and $\lens_{\cat{C},\otimes}$ are pairs $\lo{c}{x}$, where $c\in\comon_{\cat{C}}$ and $x\in\ob(\coKl(c))=\ob\cat{C}$.

The morphisms $\lo{c}{x}\to\lo{d}{y}$ in the latter are pairs $(f,f\shp)$ where $f\colon c\to d$ is a morphism of comonoids and $f\shp\colon c\otimes y\to x$ is any morphism. In the former, the morphisms are pairs $(f,f\shp)$ where $f\colon c\to d$ is a map of comonoids and $f\shp\colon\coKl(f)(y)\to x$ is a map in $\coKl(c)$. Since $\coKl(f)$ is identity on objects, $f\shp$ is a morphism $c\otimes y\to x$ in $\cat{C}$, so again the morphisms in the two categories coincide. One may check that the identities and composition formulas also coincide.
\end{proof}

\section{When $F$ is monoidal, so is $\lens_F$}

\begin{definition}[Monoidal Grothendieck construction]
Let $(\cat{C},I,\otimes)$ be a monoidal category and $(F,\varphi)\colon\cat{C}\to\smcat$ a lax monoidal functor. The monoidal Grothendieck construction \cite{moeller2018monoidal} returns a monoidal structure on the category $\gr{F}$, and hence on $\lens_F=\gr{F}\op$, by $\lo{c}{x}\otimes \lo{d}{y}\coloneqq\lo{c\otimes d}{\varphi(x,y)}$.
\end{definition}

\begin{example}
The functors $\coKl_{\cat{C}}\colon\cat{C}\to\smcat$ for arbitrary symmetric monoidal $\cat{C}$ from \cref{sec.lens_smc_v2}, as well as $\Funr{Subm}\colon\Cat{Mfd}\op\to\smcat$ and 
 $\Funr{Sh}\colon\Top\op\to\smcat$ from \cref{ex.cts_dyn,ex.ringed_sp} are all lax monoidal. Thus each of the lens categories $\lens_{\coKl}$, $\lens_{\Funr{Subm}}$, and $\lens_{\Funr{Sh}}$, inherit symmetric monoidal structures. From the first and third examples we recover the usual monoidal structure on the usual category of lenses in a symmetric monoidal category, as well as that on ringed spaces.
\end{example}

\printbibliography

\end{document}

%% file: tikz-stuff20190803.tex
\newif\ifpgfshaperectangleroundnortheast
\newif\ifpgfshaperectangleroundnorthwest
\newif\ifpgfshaperectangleroundsoutheast
\newif\ifpgfshaperectangleroundsouthwest
\pgfkeys{/pgf/.cd,
  rectangle round north east/.is if=pgfshaperectangleroundnortheast,
  rectangle round north west/.is if=pgfshaperectangleroundnorthwest,
  rectangle round south east/.is if=pgfshaperectangleroundsoutheast,
  rectangle round south west/.is if=pgfshaperectangleroundsouthwest,
  rectangle round north east, rectangle round north west,
  rectangle round south east, rectangle round south west,
}
\makeatletter
\def\pgf@sh@bg@rectangle{%
  \pgfkeysgetvalue{/pgf/outer xsep}{\outerxsep}%
  \pgfkeysgetvalue{/pgf/outer ysep}{\outerysep}%
  \pgfpathmoveto{\pgfpointadd{\southwest}{\pgfpoint{\outerxsep}{\outerysep}}}%
  {\ifpgfshaperectangleroundnorthwest\else\pgfsetcornersarced{\pgfpointorigin}\fi%
    \pgfpathlineto{\pgfpointadd{\southwest\pgf@xa=\pgf@x\northeast\pgf@x=\pgf@xa}{\pgfpoint{\outerxsep}{-\outerysep}}}}%
  {\ifpgfshaperectangleroundnortheast\else\pgfsetcornersarced{\pgfpointorigin}\fi%
    \pgfpathlineto{\pgfpointadd{\northeast}{\pgfpoint{-\outerxsep}{-\outerysep}}}}%
  {\ifpgfshaperectangleroundsoutheast\else\pgfsetcornersarced{\pgfpointorigin}\fi%
    \pgfpathlineto{\pgfpointadd{\southwest\pgf@ya=\pgf@y\northeast\pgf@y=\pgf@ya}{\pgfpoint{-\outerxsep}{\outerysep}}}}%
  {\ifpgfshaperectangleroundsouthwest\else\pgfsetcornersarced{\pgfpointorigin}\fi%
    \pgfpathclose}}

\tikzset{
  	WD/.style={%
  	  label/.style={
      	font=\everymath\expandafter{\the\everymath\scriptstyle},
        inner sep=0pt,
        node distance=2pt and -2pt},
      semithick,
      node distance=\bbx and \bby,
      decoration={markings, mark=at position \stringdecpos with \stringdec},
    	bb port length=0,
  	  bb port sep=.5,
	 	  bbx = .4cm,
		  bb min width=.4cm,
	    bby = 2ex,
	    bb rounded corners=2pt,
	    dot size=3pt,
      shell size = 16pt,
      link size = 2pt,
      shell color = blue,
      surround sep=2pt,
      ar/.style={postaction={decorate}},
  		execute at begin picture={\tikzset{
  			x=\bbx, y=\bby, 
				circuit logic US, tiny circuit symbols
				}
			}
    },
    bbx/.store in=\bbx,
    bby/.store in=\bby,
    bb port sep/.store in=\bbportsep,
    bb port length/.store in=\bbportlen,
    bb min width/.store in=\bbminwidth,
    bb rounded corners/.store in=\bbcorners,
    bb/.code 2 args={%
    	\pgfmathsetlengthmacro{\bbheight}{\bbportsep * (max(#1,#2)+1) * \bby}
      \pgfkeysalso{draw,minimum height=\bbheight,minimum
       width=\bbminwidth,outer sep=0pt,
         rounded corners=\bbcorners,thick,
         prefix after command={\pgfextra{\let\fixname\tikzlastnode}},
         append after command={\pgfextra{\draw
            \ifnum #1=0{} \else foreach \i in {1,...,#1} {
            	($(\fixname.north west)!{(2*\i-1)/(2*#1)}!(\fixname.south west)$) -- +(-\bbportlen,0) coordinate (\fixname_in\i) 
						}\fi 
            \ifnum #2=0{} \else foreach \i in {1,...,#2} {
            	($(\fixname.north east)!{(2*\i-1)/(2*#2)}!(\fixname.south east)$) -- +(\bbportlen,0) coordinate (\fixname_out\i)}\fi;
           }}}
		},
	dot size/.store in=\dotsize,
	dot/.style={
		circle, draw, thick, inner sep=0, fill=black, minimum width=\dotsize
	},
	shell size/.store in=\psize,
  spacing/.store in=\spacing,
  link size/.store in=\lsize,
  shell color/.store in=\pcolor,
 	shell inside color/.store in=\picolor,
  shell inside color=blue!20,
 	shell outside color/.store in=\pocolor,
  shell outside color=blue!50!black,
 	surround sep/.store in=\ssep,
 	link/.style={
  	circle, 
  	draw=black, 
  	fill=black,
  	inner sep=0pt, 
 		minimum size=\lsize
 	},
  shell/.style={
 		circle, 
 		draw = \pocolor, 
  	fill = \picolor,
  	minimum size = \psize
  },
  func/.style={
  	shell,
		rectangle,
		rounded corners=.5*\psize,
		inner ysep=.125*\psize,
		minimum width=1.125*\psize,
		inner xsep=.25*\psize,
  },
  funcr/.style={
    func,
    rectangle round north west=false, 
		rectangle round south west=false,
  },
  funcl/.style={
    func,
		rectangle round north east=false, 
		rectangle round south east=false,
  },
  funcu/.style={
    func,
		rectangle round south east=false, 
		rectangle round south west=false,
  },
  funcd/.style={
    func,
		rectangle round north east=false, 
		rectangle round north west=false,
  },
  outer shell/.style={
 		ellipse, 
 		draw,
  	inner sep=\ssep,
  	color=gray,
 	},
  intermediate shell/.style={
 		ellipse,
 		dashed, 
  	draw,
  	inner sep=\ssep,
 		color=\pocolor,
 	},
 }
 
\tikzset{light gray nodes/.style={every node/.style={fill=gray!40}}}